\documentclass[11pt]{amsart} 
\usepackage[english]{babel}
\usepackage[utf8]{inputenc}
\usepackage{amsmath}
\usepackage{amssymb}
\usepackage{amsfonts}
\usepackage{amsthm}
\usepackage{mathrsfs}
\usepackage[all]{xy}
\usepackage[pdftex]{graphicx}
\usepackage{color}
\usepackage{cite}
\usepackage{url}
\usepackage{indent first}
\usepackage[labelfont=bf,labelsep=period,justification=raggedright]{caption}
\usepackage[english]{babel}
\usepackage[utf8]{inputenc}
\usepackage{hyperref}
\usepackage[colorinlistoftodos]{todonotes}
\usepackage{listings}
\usepackage{tkz-fct}
\usepackage{tikz}
\DeclareMathOperator{\genus}{genus}

\DeclareMathOperator{\SL}{SL}

\DeclareMathOperator{\vol}{vol}
\newcommand{\pihol}{\pi_{\text{hol}}}
\newcommand{\piholstar}{\widehat{\pi}_{\text{hol}}}
\newcommand{\Zed}{\mathfrak{Z}}

\newcommand{\CC}{\mathbb{C}}

\newcommand{\FF}{\mathbb{F}}

\newcommand{\NN}{\mathbb{N}}
\newcommand{\PP}{\mathbb{P}}
\newcommand{\QQ}{\mathbb{Q}}
\newcommand{\RR}{\mathbb{R}}
\newcommand{\ZZ}{\mathbb{Z}}

\newcommand{\mcE}{\mathcal{E}}

\newcommand{\mcH}{\mathcal{H}}

\makeatletter
\newtheorem*{rep@thm}{\rep@title}
\newcommand{\newreptheorem}[2]{%
\newenvironment{rep#1}[1]{%
 \def\rep@title{#2 \ref{##1}}%
 \begin{rep@thm}}%
 {\end{rep@thm}}}
\makeatother

\theoremstyle{plain}
\newtheorem{thm}{Theorem}
\newreptheorem{thm}{Theorem}
\newtheorem{lemma}[thm]{Lemma}
\newtheorem{cor}[thm]{Corollary}
\newreptheorem{cor}{Corollary}

\newtheorem{prop}[thm]{Proposition}

\theoremstyle{definition}
\newtheorem{defn}[thm]{Definition}

\theoremstyle{remark}
\newtheorem*{rem}{Remark}
\newtheorem*{ex}{Example}

\numberwithin{equation}{section}
\numberwithin{thm}{section}

\begin{document}

\title{Shifted Convolution $L$-Series Values for Elliptic Curves}

\author{Asra Ali}
\address{Massachusetts Institute of Technology, Department of Mathematics, Cambridge, MA}
\email{asra@mit.edu}

\author{Nitya Mani}
\address{Stanford University,  Department of Mathematics, Stanford, CA 94305}
\email{nityam@stanford.edu}
\date{\today}


\maketitle

\begin{abstract}
Using explicit constructions of the Weierstrass mock modular form and Eisenstein series coefficients, we obtain closed formulas for the generating functions of values of shifted convolution $L$-functions associated to certain elliptic curves. These identities provide a surprising relation between weight 2 newforms and shifted convolution $L$-values when the underlying elliptic curve has modular degree $1$ with conductor $N$ such that $\text{genus}(X_0(N)) = 1$.
\end{abstract}

\section{Introduction} \label{Introduction}

The Modularity Theorem\cite{CON01} and Eichler-Shimura theory \cite{EICH57, SHI59} enable weight $2$ newforms for $\Gamma_0(N)$ with integral coefficients to be uniquely associated to isogeny classes of elliptic curves over $\QQ$. Within these results are explicit methods for constructing the $2$-dimensional complex lattice $\Lambda_E$ associated to an elliptic curve $E$, given a weight $2$ newform. 

Throughout the paper, fix the weight $2$ newform associated to an elliptic curve $E/\QQ$ as
\begin{equation}\label{e:modform}
f_E(z) = \sum_{n = 1}^{\infty} a_E(n) q^n; \quad q = e^{2\pi i z}; \qquad z \in \mcH,  \end{equation}
where $\mcH$ is the upper half-plane.
We denote the complex analytic realization of an elliptic curve as $\CC/\Lambda_E$, and its modular parametrization as $\phi_E: X_0(N) \rightarrow \CC/\Lambda_E$. The modular degree is the degree of the map $\phi_E$. In this paper, we only consider elliptic curves $E$ of modular degree $1$ with conductor $N$ such that $\text{genus}(X_0(N)) = 1$, which restricts $N$ to the finite set $\{11, 14, 15, 17, 19, 21, 27, 32, 36, 49\}$. 

Our expressions make use of the Weierstrass mock modular form associated to $E$ (a more detailed exposition can be found in Section~\ref{mockmodular}). These functions were recently introduced by Alfes, Griffin, Guerzhoy, Ono, and Rolen in \cite{ALF15,BRI13} and arise from the Weierstrass $\zeta$-function. The \emph{Weierstrass $\zeta$-function} of an elliptic curve $E$ with complex analytic realization $\CC/\Lambda_E$ is defined by
\begin{equation}\label{d:zeta}
\zeta(\Lambda_E; w) = \frac{1}{w} + \sum_{x \in \Lambda_E \setminus \{0\}} \left( \frac{1}{w-x} + \frac{1}{x} + \frac{w}{x^2} \right). 
\end{equation}
It is related to the classical Weierstrass $\wp$-function by differentiation
\[\frac{d}{dw}\zeta(\Lambda_E; w) = -\wp(\Lambda_E; w). \]
We also consider $\mathcal{E}_{f_E}(z)$, the \emph{Eichler integral} of $f_E$:
\begin{equation}\label{e:eichler}\mathcal{E}_{f_E}(z) = \sum_{n = 1}^\infty \frac{a_E(n)}{n}q^n. \end{equation}
Note that this is essentially the antiderivative of the newform $f_E$ associated to the elliptic curve $E$.

Although the Weierstrass $\wp$-function is doubly periodic with respect to $\Lambda_E$, the $\zeta$-function is not. Motivated by this, Eisenstein (see \cite{Ono09}) constructed a modification of $\zeta(\Lambda_E;w)$, a lattice invariant, non-holomorphic function $\Zed_E(w)$, defined by
\begin{equation}\label{e:zedE} \Zed_E(w) = \zeta(\Lambda_E; w) - S(\Lambda_E)w - \frac{\pi}{\vol(\Lambda_E)} \overline{w}, \end{equation}
where 
\begin{equation} \label{d:slambda} S(\Lambda_E) = \lim_{s \rightarrow 0} \sum_{x \in \Lambda_E \setminus \{0\}} \frac{1}{x^2|x|^{2s}}. \end{equation}
This value $S(\Lambda_E)$ is essentially the weight $2$ Eisenstein series for the lattice $\Lambda_E$. Using this, we define the function $\widehat{\Zed}_E(z)$ as the evaluation of $\Zed_E(z)$ at the holomorphic Eichler integral $\mathcal{E}_{f_E}(z)$,
\begin{equation}\label{e:zed}
\widehat{\Zed}_E(z) = \Zed_E(\mathcal{E}_{f_E}(z)). 
\end{equation}
Then $\widehat{\Zed}_E(z)$ is a harmonic Maa{\ss} form (see Section~\ref{mockmodular}) and can be written as a sum of a holomorphic and non-holomorphic part 
\[\widehat{\Zed}_E(z) = \widehat{\Zed}^+_E(z) + \widehat{\Zed}^-_E(z).\] We refer to the holomorphic function $\widehat{\Zed}_E^+(z)$ as the \textit{Weierstrass mock modular form} associated to $E$. When the modular degree is $1$ (as in the curves we study) $\widehat{\Zed}_E^+(z)$ is the holomorphic part of the weight $0$ harmonic Maa{\ss} form $\widehat{\Zed}_E(z)$ (see the work of Alfes, Griffin, Ono and Rolen in \cite{ALF15}), making this terminology appropriate.

The Hasse-Weil $L$-function $L(E,s) = L(f_E, s)$ plays a central role in the arithmetic of $E$. Indeed, the Birch and Swinnerton-Dyer Conjecture asserts that the arithmetic invariants of $E/\QQ$ are encoded by the analytic behavior of $L(E,s)$ at $s=1$. We can also consider \emph{shifted convolution} $L$-functions associated to an elliptic curve $E$. 
Here, we consider the shifted convolution $L$-functions evaluated at $s=1$ defined by
\begin{equation}\label{e: Dhat} D_{f_E}( h; s) = \sum_{n = 1}^{\infty} a_E(n + h) \overline{a_E(n)} \left( \frac{1}{(n+h)^s} - \frac{1}{n^s} \right).\end{equation} For convenience, we denote the generating function of these values in $h$-aspect by \begin{equation}\label{e:lser}\mathbb{L}_{f_E}(z) = \sum_{h = 1}^{\infty} D_{f_E}(h; 1)q^h.\end{equation}
Shifted convolution $L$-functions were defined by Hoffstein, Hulse, and Reznikov \cite{HOF16}. They are generalizations of the classical Rankin-Selberg convolutions \cite{RANK39}, \cite{SELB40} which were used to bound the growth of the Fourier coefficients of cusp forms. Properties of these shifted convolution Dirichlet series have been investigated recently; it was first shown that these shifted convolution values are essentially coefficients of mixed mock modular forms (see \cite{MER16}) and later the $p$-adic properties of these series (see \cite{BRI16}) and their asymptotic behavior (see \cite{BEC16}) were investigated.

Initially motivated by a desire to compute these $L$-values and understand the explicit construction of the Weierstrass mock modular form (in terms of other known elliptic invariants), we offer a closed formula for these generating functions $\mathbb{L}_{f_E}(z)$. We provide such  formulas for newforms associated to elliptic curves $E$ of modular degree $1$ and conductor $N$ such that $\text{genus}(X_0(N)) = 1$. 

These results contribute to the existing theory of shifted convolution $L$-series for the subset of curves we study. To make this more precise, let $F^{\rho(i)}_{N,2}$ be the (quasimodular) Eisenstein series for $\Gamma_0(N)$ nonvanishing and normalized to be $1$ only at the cusp $\rho(i)$ and vanishing at all other cusps (as in Section~\ref{eisenstein}).
Throughout this paper, for any $1$-periodic function $f$, let $f[h]$ denote the coefficent of $q^h$ in the Fourier expansion of $f$, and recall that $f_E = \sum_{n = 1}^{\infty} a_E(n) q^n$ denotes the weight $2$ newform associated to the elliptic curve $E$.

For the first result, we restrict to the case where $N$ is squarefree and thus $N$ lies in the set $\{11, 14, 15, 17, 19, 21\} $.



\begin{thm} \label{t:thm1}
Assume the notation and hypotheses above. Then, we have that

$$\mathbb{L}_{f_E}(z) = \frac{\vol(\Lambda_E)}{\pi} \left( (f_E(z) \cdot \widehat{\Zed}_E^+(z)) - \alpha f_E(z) - F^{\infty}_{N,2}(z) \right),$$ where $$\alpha = (f_E \cdot \widehat{\Zed}_E^+)[1] - \frac{\pi}{\vol{\Lambda_E}}D_{f_E}(1; 1) - F^{\infty}_{N,2}[1].$$ 
\end{thm}

A special case occurs when we drop the condition that $N$ is squarefree and instead suppose that $E$ has complex multiplication, restricting $N$ to be in the set $\{27, 32, 36\}$ (the result does still hold for $N=49$ numerically, but is excluded from the this set due to a technical limitation). In this case, computational evidence proves that $\alpha = 0$, yielding the following stronger result.

\begin{thm}\label{t:thm1.2}
Let $E$ be an elliptic curve as above with complex multiplication and conductor $N \neq 49$. With $\mathbb{L}_{f_E}(z)$ defined as above, we have
\[\mathbb{L}_{f_E}(z) = \frac{\vol(\Lambda_E)}{\pi}\big((f_E(z) \cdot \widehat{\Zed}_E^+)(z) - F^\infty_{N,2}(z)  \big). \]
\end{thm}

These identities use the characterization of holomorphic projection given by Mertens and Ono in \cite{MER16} to explicitly relate the Fourier coefficients of weight $2$ newforms to elliptic curve invariants and to understand the behavior of the Weierstrass mock modular form. These identities can also be used to compute the shifted convolution $L$-values to arbitrary precision. 

In the following section, we provide some preliminaries on harmonic Maa\ss\,forms, mock modular forms, and characterize the Weierstrass mock modular form. We also introduce Maa\ss-Poincar\'e series and holomorphic projection which are connected to the shifted convolution $L$-series we study. This enables us in Section 3 to compute the generating function for the shifted convolution $L$-series in the desired cases, as well as prove vanishing properties of the shifted convolution Dirichlet series and holomorphic projection associated to a subset of elliptic curves. Finally, we give examples of our explicit results in the case of conductors $N = 11$ and $N = 27$.




\section{Preliminaries}\label{prelim}
\subsection{Harmonic Maa{\ss} forms and mock modular forms}
We begin with a review of harmonic Maa{\ss} forms. These real-analytic modular forms were first introduced by Bruinier and Funke in \cite{BRU04}. Among many other important roles these forms play in number theory, work by Zwegers (see \cite{ZW03}) shows that Ramanujan's mock $\theta$-functions arise as holomorphic parts of harmonic Maa{\ss} forms. These forms are also intimately connected to our study of the holomorphic projections associated to certain elliptic curves.

Let $\mcH = \{x + iy : x, y \in \RR,\, y > 0\}$ denote the upper half-plane. Consider  $z = x+iy \in \mcH$ and as noted in the introduction, let $q = e^{2\pi i z}$. For any $\gamma = \begin{bmatrix}
a & b \\
c & d \\
\end{bmatrix} \in \SL_2(\ZZ)$ we abbreviate the action of $\gamma$ on a function $f$ by the Petersson slash operator of weight $k$, $$(f |_k \gamma)(z) = (cz + d)^{-k} f \left(\frac{az + b}{cz + d} \right).$$
Consider the congruence subgroup of level $N$, defined as follows: $$\Gamma_0(N) = \left\{ \begin{bmatrix}
a & b \\
c & d \\
\end{bmatrix} \in \SL_2(\ZZ) \,|\, c \equiv 0 \hspace{-3pt} \mod N
\right \}.$$

\begin{defn}(\hspace{-4pt} \cite{BRU04}) \label{d:harmonicmaass}
A \textit{weak harmonic Maa{\ss} form} of weight $k \in \ZZ$ on $\Gamma_0(N)$ is a smooth function $f$ on $\mcH$ that satisfies the following three conditions:
\begin{enumerate}
\item $(f |_k \gamma)(z) = f(z)$ for all $\gamma \in \Gamma_0(N)$ (i.e. $f$ transforms like a modular form on $\Gamma_0(N)$),
\item $\Delta_k f(z) \equiv 0$, where if $z = x + iy$, $\Delta_k = -y^2 \left( \frac{\partial^2}{\partial x^2} + \frac{\partial^2}{\partial y^2} \right) +  i k y \left( \frac{\partial}{\partial x} + i \frac{\partial}{\partial y} \right)$
\item $f(z)$ has poles at most at the cusps of $\Gamma_0(N)$ (i.e. there exists some polynomial $P(z)$, such that $f(z) - P(q^{-1}) = O(e^{- \alpha y})$ for $\alpha > 0$ as $y \rightarrow \infty$ and an analogous condition holds for the other cusps of $\Gamma_0(N)$).
\end{enumerate}
\end{defn}

Note that hereafter, we will simply term such smooth $f$ as weight $k$ harmonic Maa{\ss} forms for $\Gamma_0(N)$. Following immediately from the definition above, we see that a harmonic Maa{\ss} form admits a  Fourier expansion of the following form: 

\begin{lemma}[\cite{BRU04}]
A weight $k$ harmonic Maa{\ss} form $f(z)$ admits a Fourier expansion of the form,
$$f(z) = f^+(z) + f^-(z);$$
$$f^+(z) = \sum_{n \gg -\infty} c_f^+(n) q^n; \qquad f^-(z) = \sum_{n = 1}^{\infty} c_f^-(n) q^n \Gamma(k-1, 4 \pi n y).$$
Here, $\Gamma(\alpha, \beta) = \int_\beta^\infty e^{-t}t^{\alpha - 1}dt$ denotes the incomplete Gamma function, and $f^+(z)$ and $f^-(z)$ are the holomorphic and non-holomorphic parts of $f(z)$ respectively.
\end{lemma}
This construction also gives rise to a characterization of mock modular forms in terms of harmonic Maa{\ss} forms:

\begin{defn}
Consider $f(z) = f^+(z) + f^-(z)$ as defined above. If $f^-(z) \neq 0$, then $f^+(z)$ is called a \emph{mock modular form}.  
\end{defn}
We define a differential operator (as in \cite{Ono09}) to help characterize $f^-(z)$ when it is nontrivial. Denote by $H_{k}(\Gamma)$ the space of weight $k$ harmonic Maa{\ss} forms for some congruence subgroup $\Gamma \le \SL_2(\ZZ)$. 

\begin{prop}[\cite{BRU04}]\label{p:xi}
Define a differential operator $\xi_{2-k}: H_{2-k}(\Gamma) \rightarrow S_k(\Gamma)$ where $\Gamma \le \SL_2(\ZZ)$ as $$\xi_{2-k}(f(z)) = 2 i y^k \frac{\overline{\partial f}}{\partial \overline{z}}.$$
Then $\xi_{2-k}$ is a well-defined, surjective, antilinear map with kernel $M_{2-k}^!(\Gamma)$, the space of weakly holomorphic weight $2-k$ modular forms for $\Gamma$.
Further, $$\xi_{2-k} f(z) = -(4\pi)^{k-1} \sum_{n = 1}^{\infty} c_f^-(n) q^n.$$
\end{prop}

We call the function defined by $\sum_{n = 1}^{\infty} c_f^-(n) q^n$ the \emph{shadow} of $f(z)$ or $f^+(z)$. When the shadow is non-trivial, we obtain some conditions on the cusp behavior of $f(z)$.

\begin{prop} (Lemma 2.3 of~\cite{BR12}) \label{p:nonconstant}
If $f(z) \in H_{2-k}(\Gamma_0(N))$ has the property that $\xi_{2-k}(f) \neq 0$, then the principal part of $f(z)$ is nonconstant for at least one cusp.
\end{prop}


The kernel of the surjective $\xi_{2-k}$ operator defined above is infinite dimensional. Selecting a suitable class of harmonic Maa{\ss} forms to serve as preimages under this $\xi$ operator depends intimately on the following notion:
\begin{defn}
A harmonic Maa{\ss} form $F(z) \in H_{2-k}(\Gamma_0(N))$ is \textit{good} for $f(z) \in S_k(\Gamma_0(N))$ if it satisfies the following $3$ conditions:
\begin{enumerate}
\item The principal part of $F(z)$ at the cusp $\infty$ is in $\FF_f[q^{-1}]$, where $\FF_f$ is the field obtained by adjoining the Fourier coefficients of $f$ to $\QQ$.
\item The principal part of $F(z)$ at all other inequivalent cusps of $\Gamma_0(N)$ is constant.
\item $\xi_{2-k}(F(z)) = \| f(z)\|^{-2} f(z)$, where $\xi_{2-k}$ is the differential operator defined in Proposition~\ref{p:xi}.
\end{enumerate}
\end{defn}

\subsection{Weierstrass mock modular forms}\label{mockmodular}

We begin with the construction of the Weierstrass mock modular form associated with an elliptic curve $E$ of conductor $N$. This will prove very useful to our analysis of shifted convolution $L$-series values. The Weierstrass mock modular form was introduced in \cite{BRI13} as a mechanism by which to understand properties of elliptic curves and their associated newforms through the language of harmonic Maa{\ss} forms. The Weierstrass mock modular form has been one of the primary objects of study by many recently, as in \cite{ALF15,CL16}.

The theory of elliptic curves gives rise to a notable example of a weight $0$ harmonic Maa{\ss} form. Recall that $E \simeq \mathbb{C}/ \Lambda_E$ where $\Lambda_E$ is a $2$-dimensional lattice in $\mathbb{C}$. Recall the Weierstrass $\zeta$-function defined in~\eqref{d:zeta}. Although it is not elliptic, its derivative is negative the Weierstrass $\wp$-function. This relation gives the Laurent expansion of $\zeta$:

\begin{prop}\cite{DIA06}\label{p:Laurent}
The Laurent expansion of $\zeta$ is
\[ \zeta(\Lambda_E; z) = \frac{1}{z} - \sum_{k = 1}^\infty G_{k+2}(\Lambda_E) z^{2k+1},\]
where $G_k(\Lambda_E)$ is the Eisenstein series of weight $k$ associated to a lattice $\Lambda_E$.
\end{prop}



Again, recall the construction of the Weierstrass mock modular form given in \eqref{e:zedE}:
\begin{align*}
\widehat{\Zed}^+_E(z) = \Zed^+_E(\mcE_{f_E}(z)) &= \zeta(\Lambda_E;\mcE_{f_E}(z)) - S(\Lambda_E)\mcE_{f_E}(z). 
\end{align*}

The following theorem outlines some important properties about the Weierstrass mock modular form.
\begin{thm}(\hspace{-4pt} \cite{ALF15})
Assume the notation and hypotheses above. Then
\begin{enumerate}
\item The holomorphic part $\widehat{\Zed}^+_E(z) = \zeta(\Lambda_E;\mathcal{E}_{f_E}(z)) - S(\Lambda_E)\mathcal{E}_{f_E}(z)$ has poles exactly when $\mathcal{E}_{f_E}(z)$ is a lattice point.
\item If $\widehat{\Zed}^+_E(z)$ has poles in the upper half plane, there is a canonical meromorphic modular function $M_E(z)$ such that $\widehat{\Zed}^+_E(z) - M_E(z)$ is holomorphic on $\mcH$.
\item $\widehat{\Zed}_E(z) - M_E(z)$ is a harmonic Maa{\ss} form of weight $0$ on $\Gamma_0(N)$ and $\xi_0(\widehat{\Zed}^+_E(z) - M_E(z)) = -(4\pi)\cdot f_E(z)$. In particular, $\widehat{\Zed}^+_E(z) - M_E(z)$ is a weight $0$ mock modular form.
\end{enumerate}
\end{thm}
In particular, $\widehat{\Zed}_E^+(z) - M_E(z)$ is called the Weierstrass mock modular form of $E$.



In particular, we will be interested in the case where $\widehat{\Zed}^+_E(z)$ itself is a mock modular form of weight $0$. This arises when the canonical $M_E(z)$ is identically zero. The following proposition gives sufficient conditions for this to occur.

\begin{lemma}\label{p:pole}
Let $E$ be an elliptic curve of conductor $N$ with modular parametrization $\phi_E: X_0(N) \rightarrow E$. If the modular degree $\deg(\phi_E)$ is $1$, then $\widehat{\Zed}^+_E(z)$ does not have poles in the upper half-plane.
\end{lemma}

\begin{proof}
Recall the mock modular form $\widehat{\Zed}_E^+(z)$ is given by Equation~\eqref{e:zedE} with Laurent expansion as in Proposition~\ref{p:Laurent}:
\begin{align*} \widehat{\Zed}_E^+(z) &= \Zed_E(\mathcal{E}_{f_E}(z)) \\
&= \frac{1}{\mathcal{E}_{f_E}(z)} - \sum_{k = 1}^\infty G_{2k+2}(\Lambda_E)\mathcal{E}_{f_E}(z)^{2k+1} - S(\Lambda_E)\mathcal{E}_{f_E}(z).
\end{align*}
The Eichler integral $\mathcal{E}_{f_E}(z) = \sum_{n = 1}^\infty \frac{a_E(n)}{n}q^n$ is holomorphic on the upper half-plane, so it suffices to show that $\mathcal{E}_{f_E}(z)$ does not vanish for any $z \in \mcH$. 
The modular parametrization $\phi_E: X_0(N) \rightarrow \CC/\Lambda_E$ is induced from the map $\phi_1: \mcH \rightarrow \CC$. Here $\phi_1$ is the map given by
\begin{align*} \phi_1(z) &= - 2\pi i \int_z^{i\infty} f_E(\tau)d\tau \\
&= \sum_{n = 1}^\infty \frac{a_E(n)}{n}q^n = \mathcal{E}_f(z). \end{align*}
If the modular degree $\deg(\phi_E)$ is $1$ (requiring that the genus of $X_0(N)$ is $1$, since $\deg(\phi_E) \geq \text{genus}(X_0(N))$, then the map $\phi_E$ is an isomorphism. Thus, $\mathcal{E}_{f_E}(z)$ does not vanish for any $z \in \mcH$ since it vanishes at the cusp $\infty$.
\end{proof}


Moreover, $\widehat{\Zed}^+_E(z)$ has rational Fourier coefficients if $E$ has complex multiplication (see Theorem 1.3 in \cite{BRU08}).

\begin{ex} Consider the strong Weil curve of conductor $27$ given by the Weierstrass equation $E_{27}: y^2 + y = x^3 - 7$ (Cremona label 27a1).
The weight $2$ modular form associated with $E_{27}$ is given by
\[f_{E_{27}} =  q - 2q^{4} - q^{7} + 5q^{13} + 4q^{16} - 7q^{19} + O(q^{20})  .\]
Using the Fourier expansion of the Weierstrass $\zeta$-function, the weight $0$ mock modular form associated to $\widehat{\Zed}_{E_{27}}^+(z)$ is given by 
\begin{equation} \label{e:zed27} \widehat{\Zed}_{E_{27}}^+(z) = q^{-1} + \frac{1}{2}q^2 + \frac{1}{5}q^5 + \frac{3}{4}q^8 - \frac{6}{11}q^{11} - \frac{1}{2}q^{14} + O(q^{17}) .\end{equation}  
\end{ex}

\subsection{Poincar\'e series}


The modular parametrization of an elliptic curve $E$ is given by a map $\phi_E: X_0(N) \rightarrow E$ where $X_0(N)$ is the compactification of the curve $\Gamma_0(N) \setminus \mcH$. Let $f_E$ be the weight $2$ newform associated to this parametrization. The Petersson norm of $f_E$ is then
\[\|f_E \|^2 = \langle f_E, f_E \rangle = \int_{z \in \Gamma_0(N) \setminus \mcH} |f_E(z)|^2dx \wedge dy. \]

Using Petersson norms, we can relate the degree of $\phi_E$ with the area of the fundamental parallelogram of the period lattice $\Lambda_E$, which is the volume of the elliptic curve, $\vol(\Lambda_E)$.

\begin{prop}[\hspace{-4pt} \cite{ZAG85}]
The volume $\vol(\Lambda_E)$ of an elliptic curve $E$ is
\[\vol(\Lambda_E) = \frac{4\pi^2 \| f_E\|^2}{\deg(\phi_E)}. \]
\end{prop}

The Petersson inner product can also be used to extract Fourier coefficients of cusp forms through Poincar\'e series. 

A generic index $m$ Poincar\'e series is given by
\begin{equation}
\mathbb{P}(m,k,\phi_m,N;z) = \sum_{\gamma \in \Gamma_\infty \setminus \Gamma_0(N)} (\phi_m^* |_k \gamma)(z),
\end{equation}
where $\Gamma_\infty = \{ [\begin{smallmatrix} 1 & n \\ 0 & 1 \end{smallmatrix}] : n \in \ZZ \}$ and $\phi_m^*(z) = \phi_m(y)e^{2\pi i m x}$ for a function $\phi_m: \RR_{> 0} \rightarrow \CC$ which satisfies $\phi_m(y) = O(y^\alpha)$ as $y \rightarrow 0$ for some $\alpha \in \RR$.

Using this, the classical index $m$ Poincar\'e series $P(m, k, N; z)$ and the Maa{\ss}-Poincar\'e series $Q(-m,k,N;z)$ are defined as
\begin{align}
P(m,k,N;z) &= \mathbb{P}(m,k,e^{-my}, N;z), \\
Q(-m, k, N; z) &= \frac{1}{(k-1)!}\mathbb{P}(-m, 2-k, N, \mathcal{M}_{1-\frac{k}{2}}(-4\pi m y); z),
\end{align}
where $\mathcal{M}_s(y)$ is defined in terms of the $M$-Whittaker function
\[\mathcal{M}_s(y) = |y|^{-\frac{k}{2}}M_{\frac{k}{2}\text{sgn}(y),s-\frac{1}{2}}(|y|) \]
defined in \cite{WHI03}.

We can characterize a set of Maa{\ss}-Poincar\'e series in terms of the above functions, with a Fourier expansion given using Bessel functions and the Kloosterman sum $K(m, n; c)$ we recall below:

$$K(m, n;c) = \sum_{d \textnormal{ mod }c; (c, d) = 1} e^{ 2\pi i \frac{m\overline{d} + nd}{c} },$$ where $\overline{d}$ is the multiplicative inverse of $d$ modulo $c$.
We first recall the Fourier expansion of the classical Poincar\'e series:
\begin{prop}[Theorem 8.3 \cite{Ono09}]\label{p:petFourier}
Consider the weight $k$ Poincar\'e series of index $m$ and level $N$, $P(m, k, N; z)$. Then this Poincar\'e series has a Fourier expansion as $P(m, k, N; z) = q^m + \sum_{n = 1}^{\infty} b_P(m, k, N;n) q^n$ where $b_P(m,k,N;n)$ can be defined as follows (when $k \equiv 0 \pmod{2}$):
$$b_P(m,k,N;n) = \left(\frac{n}{m}\right)^{(k-1)/2} \left(\delta_{m, n} + 2\pi i^{-k} \sum_{c > 0; N | c} J_{k -1} \left( \frac{2\pi \sqrt{mn}}{c} \right) \frac{K(m, n; c)}{c} \right).$$ 
\end{prop}

The Fourier expansion and behavior of the Poincar\'e series at the cusps of $\Gamma_0(N)$ is explained by the following characterization of the Maa\ss-Poincar\'e series:

\begin{prop}(6.2 in \cite{BRU08}, 3.3 in \cite{RH12})\label{p:poincusp}
If $k \in 2\NN$, and $m, N \geq 1$, then $Q(-m, k, N;z) \in H_{2-k}(\Gamma_0(N))$, and has a Fourier expansion of the form
\[Q(-m,k,N;z) = Q^+(-m,k,N;z) + Q^-(-m, k, N; z),\]
where 
\[Q^+(-m, k, N;z) = q^{-m} + \sum_{n = 0}^\infty b_Q(-m,k,N;n)q^n \]
and for integers $n \geq 0$ we have
\begin{align*}
b_Q(-m, k, N; n) &= -2\pi (-1)^{k/2} \cdot \sum_{\substack{c > 0 \\ c \equiv 0 \pmod N}} \left( \frac{m}{n}\right)^{\frac{k-1}{2}} \frac{K(-m, n, c)}{c} \cdot I_{k-1}\left(\frac{4\pi \sqrt{|mn|}}{c} \right), \\
b_Q(-m, k, N; 0) &= -\frac{2^k \pi^k (-1)^\frac{k}{2} m^{k-1}}{(k-1)!} \cdot \sum_{\substack{c > 0 \\ c \equiv 0 \pmod N}} \frac{K(-m, 0, c)}{c^k}.
\end{align*}
Additionally, the principal part at all other cusps is zero.
\end{prop}

The classical Poincar\'e series $P(m, k, N;z)$ and the Maa{\ss}-Poincar\'e series are also related by the differential operator $\xi_{2-k}$ in the following proposition.

\begin{prop}(2.6 in \cite{MER16})
If $k \geq 2$ is even and $m, N \geq 1$, then 
\[\xi_{2-k}(Q(-m, k, N; z)) = (4\pi)^{k-1} m^{k-1} (k-1) \cdot P(m, k, N; z) \in S_k(\Gamma_0(N)). \]
\end{prop}

\subsection{Holomorphic projection and shifted convolution Dirichlet series}\label{holproj}

The study of holomorphic projection is motivated by a desire to understand smooth functions $f$ which transform like modular forms and have ``moderate growth'' at cusps. A function $f$ with such properties defines a linear functional on the space of cusp forms via the Petersson inner product, and thus we can associate to to $f$ the cusp form defining the same linear functional as $f$. This cusp form is essentially the holomorphic projection of $f$. The holomorphic projection was first introduced in \cite{STU80} by Sturm and further developed in the work of Gross and Zagier in \cite{GRO86}. We will use it here to give a closed-form algebraic characterization of the shifted convolution $L$-series values for some modular forms associated to elliptic curves. Our approach follows previous work as in \cite{MER16, BRI16}.

We define the holomorphic projection for continuous functions $f: \mcH \rightarrow \CC$, where $f(z) = \sum_{n \in \ZZ} a(n , y)q^n$ that transform like a modular form of weight $k \ge 2$ for $\Gamma_0(N)$ and have moderate growth at the cusps. We can make this idea more precise. Suppose the cusps of $\Gamma_0(N)$ are $\rho(i)$ (where $\rho(1)$ is chosen to be the infinite cusp), and consider $\sigma_i \in \SL_2(\ZZ)$ so that $\sigma_i \infty = \rho(i)$. Then $f$ has \emph{moderate growth at the cusps} if for $n > 0$ $$a(n, y) = O(y^{2-k}), \quad y \rightarrow 0,$$ and 
$$f |_k\, \sigma_i = c_0^{(i)} + O \left(\frac{1}{d(y)} \right), \quad y \rightarrow \infty,$$ 
where $c_0^{(i)}$ is the value of the constant Fourier coefficient when $f$ is evaluated at the cusp denoted $\rho(i)$ and $d(y)$ is some polynomial in $y$.

\begin{defn}[\hspace{-4pt} \cite{STU80}]
Consider a continuous function $f(z) = \sum_{n \in \ZZ} a(n , y)q^n$ as above (where $z = x + iy$) that transforms like a modular form of weight $k \ge 2$ for $\Gamma_0(N)$ and has moderate growth at the cusps. Then, the \textit{holomorphic projection} of $f(z)$, denoted $\pihol(f)(z)$ is constructed as follows: $$\pihol(f)(z) = c_0^{(1)} + \sum_{n = 1}^{\infty} c(n) q^n;$$ $$c(n) = \frac{(4 \pi n)^{k-1}}{(k-2)!} \int_{0}^{\infty} a(n, y) e^{-4 \pi n y}y^{k - 2} dy.$$
\end{defn}

The holomorphic projection satisfies several natural properties that arise from its construction:

\begin{prop}[\hspace{-4pt} \cite{GRO86}]\label{p:holprops}
Consider $f$ as defined above. Then the holomorphic projection of $f$ satisfies the following three properties. 
\begin{enumerate}
\item If $f$ is a holomorphic modular form, $\pihol(f) = f$
\item If $k > 2$, $\pihol(f) \in M_k(\Gamma_0(N))$, the space of weight $k$ modular forms for $\Gamma_0(N)$. If $k = 2$, $\pihol(f) \in M_2(\Gamma_0(N)) \oplus \CC E_2 = \widetilde{M_2}(\Gamma_0(N))$, the space of weight $2$ quasimodular forms for $\Gamma_0(N)$. 
\item $\langle g, f \rangle = \langle g, \pihol(f) \rangle$ for any $g \in S_k(\Gamma_0(N)).$ 
\end{enumerate}
\end{prop}

We can understand the holomorphic projection more explicitly in specific cases, such as the following product of a harmonic Maa{\ss} form and a cusp form.

\begin{prop}[\hspace{-4pt} \cite{MER16}] \label{p:hol}
Let $M_{f_1}$ be the weight $2-k$ harmonic Maa{\ss} form whose shadow is $f_1 \in S_k(\Gamma_0(N))$ where $f_1(z) = \sum_{n = 1}^{\infty} a_1(n)q^n$, so $\xi_{2 - k} M_{f_1} = -(4 \pi)^{k-1} f_1$. Consider also the weight $k$ cusp form for $\Gamma_0(N)$, $f_2(z) = \sum_{n = 1}^{\infty} a_2(n)q^n$. Suppose that $M_{f_1}^+ \cdot f_2(z)$ has moderate growth at all cusps. Then, \begin{equation}\label{e:piholMff} \begin{split} \pihol (M_{f_1} \cdot f_2)(z) & = M_{f_1}^+ (z)\cdot f_2(z)   \\ & - (k - 2)! \sum_{h = 1}^{\infty} \left[ \sum_{n = 1}^{\infty} a_2(n+h)  \overline{a_1(n)} \left( \frac{1}{(n + h)^{k-1}} - \frac{1}{n^{k-1}} \right) \right] q^h. \end{split}\end{equation}
\end{prop}

Now consider a strong Weil curve $E$ with associated weight $2$ newform $f_E$ for $\Gamma_0(N)$, where the genus of the modular curve $X_0(N)$ is $1$, and let $\widehat{\Zed}_E$ be defined as in~\eqref{e:zed}. Then, we can compute the holomorphic projection of $f_E \cdot \widehat{\Zed}_E$ as follows: 

\begin{cor}\label{p:picomputation}
Let $E$ be a strong Weil curve with associated weight $2$ newform $f_E$ for $\Gamma_0(N)$, where $\genus(X_0(N)) = 1$ and $\deg(\phi_E) = 1$, and let $\widehat{\Zed}_E$ be defined as in~\eqref{e:zed}. Then, we have the following: 
$$\pihol(f_E \cdot \widehat{\Zed}_E)(z) = \frac{\vol(\Lambda_E)}{\pi} f_E(z) \widehat{\Zed}_E^+(z) - \sum_{h = 1}^{\infty}  D_{f_E}(h; 1) q^h.$$
\end{cor}
\begin{proof}
Since $\text{dim}(S_2(\Gamma_0(N))) = \text{genus}(X_0(N)) = 1$, the modular form $f_E$ is a scalar multiple of the Poincaré series $f_E = \frac{1}{\beta} P(1, 2, N;z)$, where $P(m, k, N;z)$ is the Poincaré series described in Proposition~\ref{p:petFourier}. Then the Petersson coefficient formula (see \cite{MER16}) and~\eqref{d:slambda} yields $\beta = \frac{\pi}{\vol(\Lambda_E)}$. Following the computation in Corollary 1.2 of~\cite{MER16}, we obtain the holomorphic projection in terms of the Poincaré series
\[\pihol(f_E \cdot \widehat{\Zed}_E)(z) = P(1, 2, N; z) \widehat{\Zed}_E^+(z) - \mathbb{L}_{f_E}(z). \]
We can apply Proposition~\ref{p:hol} with $P(1,2,N; z) = \frac{1}{\beta} f_E$, which yields the Lemma.
\end{proof}

\begin{rem}
For the remainder of the paper, we will define a new function $\piholstar$, a scalar multiple of the holomorphic projection by the constant  $\frac{\pi}{\vol(\Lambda_E)}$ for ease of algebraic computation and numerical characterization. Thus, we will say $$\piholstar(f_E \cdot \widehat{\Zed}_E)(z) =  f_E(z) \widehat{\Zed}_E^+(z) - \frac{\pi}{\vol(\Lambda_E)} \mathbb{L}_{f_E}(z)$$
\end{rem}

\subsection{Eisenstein series}\label{eisenstein}

In order to understand $\piholstar(f_E \cdot \widehat{\Zed}_E)$, we define a basis for the space of \emph{weight $2$ quasimodular forms} for $\Gamma_0(N)$, the space $\mcE_2(\Gamma_0(N)) \oplus \CC E_2$, denoted $\widetilde{\mcE_2}(\Gamma_0(N))$. To do this, we follow the construction given in \S2 of Chapter VII in \cite{SCH12} and arrive at a set of forms $F_{N, 2}^{-a_2/a_1}$ described below:

\begin{defn} 
Define the Eisenstein series $G_{N, k}^{-a_2/a_1}: \mcH \rightarrow \widehat{\CC}$ for $\Gamma(N), k>2$, as $$G_{N, k}^{-a_2/a_1}(\tau) =  \sideset{}{'}\sum_{m_1 \equiv a_1 (N), m_2 \equiv a_2 (N)} (m_1\tau + m_2)^{-k}$$ where the sum is taken over nonzero integer pairs $(m_1, m_2)$ satisfying the congruence conditions described above, where $a = (a_1, a_2) \in \PP^1(\QQ)$ ranges over the $\Gamma(N)$  inequivalent cusps. 
Let $$\phi^{-a_2/a_1}_{N, k}(\tau, s) = \sideset{}{'}\sum_{m_1 \equiv a_1 (N), m_2 \equiv a_2 (N)} (m_1 \tau + m_2)^{-k} |m_1 \tau + m_2 |^{-s}$$ with analytic continuation to the $s$-plane as a meromorphic function $\phi_{N, k}^{-a_2/a_1}$ (that is holomorphic at $s = 0$). Then for $k = 2$, define $$G_{N, 2}^{-a_2/a_1}(\tau) =  \phi_{N, 2}^{-a_2/a_1}(\tau, 0)$$ where $(a_1, a_2), N$ are as above.
\end{defn}

We can also associate an Fourier expansion to each Eisenstein series $G_{N, k}^{-a_2/a_1}$ and note that each of these defines a modular form.

\begin{prop}[Chapter VII \cite{SCH12}]\label{p:fourierg}
For all $a = (a_1, a_2)$, $N \ge 2$, $k > 2$ $G_{N, k}^{-a_2/a_1}$ is a (holomorphic) modular form weight $k$ for $\Gamma(N)$ with a Fourier expansion given as follows, where $\zeta_N = e^{2\pi i /N}$ and $\delta(a_1/N) = 1$ if $a_1 \equiv 0 \pmod N$ else is $0$: $$G_{N, k}^{-a_2/a_1}(\tau) = \delta \left( \frac{a_1}{N} \right)  \sideset{}{'}\sum_{m_2 \equiv a_2 (N)} m_2^{-k} + \sum_{n > 0} \alpha_n(N, k, a) \exp(2 \pi i \tau n / N)$$ $$\alpha_n = \frac{(-2\pi i)^k}{N^k (k - 1)!} \sum_{m | n, \frac{n}{m} \equiv a_1 \pmod{N}} m^{k-1} \textnormal{sign} (m) \zeta_N^{a_2 m} \quad n > 0$$ 
If $k = 2$ and $N > 2$, the Fourier expansion of $G_{N, k}^{-a_2/a_1}(\tau)$ is given as follows (where $\alpha_n$ is defined similarly: $$G_{N, 2}^{-a_2/a_1}(\tau) = - \frac{2\pi i}{N^2(\tau - \overline{\tau})} + \delta \left( \frac{a_1}{N} \right)  \sideset{}{'}\sum_{m_2 \equiv a_2 (N)} m_2^{-2} + \sum_{n > 0} \alpha_n(N, 2, a) \exp(2 \pi i \tau n / N)$$ Note that $G_{N, 2}^{-a_2/a_1}(\tau)$ transforms like a modular form although it is not holomorphic
\end{prop}

Given this Fourier expansion, we can construct a basis for $\mcE_k(\Gamma(N))$ comprised of linear combinations of $G_{N, k}^{-a_2/a_1}$ such that for each $\Gamma(N)$-inequivalent cusp, exactly one element of the basis is nonvanishing. Using this above spanning set for $\widetilde{\mcE}_2(\Gamma(N))$ we give a basis for $\widetilde{\mcE}_2(\Gamma_0(N))$. In doing so, we will follow the construction for Eisenstein series for arbitrary congruence subgroups given in \cite{SCH12}. 

The dimension of this space $\dim(\mcE_2(\Gamma_0(N))$ is $\sigma_\infty - 1$ where $\sigma_{\infty}(N)$ is the number of $\Gamma_0(N)$ inequivalent cusps\cite{DIA06}. Let $$\widetilde{\mcE}_2(\Gamma_0(N)) = \CC E_2 \oplus \mcE_2(\Gamma_0(N))$$ so that $$\dim(\widetilde{\mcE}_2(\Gamma_0(N)) = \sigma_{\infty}(N).$$

\begin{defn}
Suppose that $[\Gamma_0(N): \Gamma(N)] = d$ and denote by $A_1, ... A_d$ a set of coset representatives of $\Gamma_0(N)$ in $\Gamma(N)$. Then define a set of Eisenstein series for $\Gamma_0(N)$ by averaging over these coset representatives as follows: \begin{align}\label{e:cosetsum} H_{\Gamma_0(N), k}^{-a_2/a_1} = \sum_{n = 1}^d H_{N, k}^{-a_2/a_1} |_k A_n. \end{align}
\end{defn}

\begin{lemma}
The Eisenstein series $H_{\Gamma_0(N), 2}^{-a_2/a_1}$ is nonvanishing at all cusps $\Gamma_0(N)$-equivalent to $-a_2/a_1$ and vanishes at all other cusps.
\end{lemma}
\begin{proof} The functions $H_{\Gamma_0(N), 2}^{-a_2/a_1}$ are given explicitly in \S2 in \cite{SCH12} as an average of functions $G_{N,2}^{-a_2/a_1}$ over pairs $(a_1, a_2)$. Following the Fourier expansions given in \S2.2 in \cite{SCH12}, the series $G_{N, 2}^{-a_2/a_1}(\tau)$ is identical to $G_{N,k}^{-a_2/a_1}(\tau)$ for $k > 2$ except for a non-holomorphic contribution $\frac{-2\pi i}{N^2(\tau - \overline{\tau})}$. Because this component is independent of $a = (a_1, a_2)$, its contribution to the averaged function $H_{\Gamma_0(N), k}^{-a_2/a_1}$ from Equation~\eqref{e:cosetsum} is also independent of $a$. The proof then follows \emph{mutatis mutandis} as in Theorem 4, Ch. 7 of \cite{SCH12}.
\end{proof}


\begin{lemma}\label{l:basis}
If $a = (a_1, a_2) = (0, -1)$ corresponds to the cusp $\infty$, denoted by $H_{\Gamma_0(N), 2}^{\infty}$, then set $\{H_{\Gamma_0(N), 2}^{-a_2/a_1}\}$ as $a = (a_1, a_2)$ ranges over all cusps of $\Gamma_0(N)$ forms a basis for $\widetilde{\mcE}_2(\Gamma_0(N))$.
\end{lemma}
\begin{proof}
Since $\dim(\widetilde{\mcE}_2(\Gamma_0(N))) = \sigma_{\infty}(N)$, the set $\{H_{\Gamma_0(N), 2}^{-a_2/a_1}\}$ forms a basis for $\widetilde{\mcE}_2(\Gamma_0(N))$ as desired.
\end{proof}

Using this construction of $H_{\Gamma_0(N), k}^{-a_2/a_1}$, we can construct a normalized basis for $\widetilde{\mcE}_2(\Gamma_0(N))$:


\begin{prop}
There is a set of weight $2$ quasimodular Eisenstein forms for $\Gamma_0(N)$, say $F_{N,2}^{-a_2/a_1}$, such that if $\gcd(a_1, a_2) = 1$, then $F_{N, 2}^{-a_2/a_1}$ is $1$ at the cusps $\Gamma_0(N)$-equivalent to the rational $-a_2/a_1$ and $0$ at the other cusps. The set of linearly independent $F_{N,2}^{-a_2/a_1}$ forms a basis for $\widetilde{\mcE}_2 (\Gamma_0(N))$. 
\end{prop}

Using these Eisenstein series, we can give an alternate representation of the holomorphic projection described in the previous section using Propstion~\ref{p:hol}. This occurs because $\widehat{\pi}_{hol}(f_E \cdot \widehat{\Zed}_E)$ lies in the space of weight $2$ quasimodular forms for $\Gamma_0(N)$ as per Proposition~\ref{p:holprops}, and $f_E \cdot \widehat{\Zed}_E$ has moderate growth at all cusps (which will be shown later in Lemma~\ref{p:computeInvolution}).

\begin{cor}\label{l:holprojRHS}
Consider a strong Weil curve $E$ with conductor $N$, associated modular form $f_E$, and $\widehat{\Zed}_E$ as defined in~\eqref{e:zedE}. Then there are numbers $\alpha, \beta_1, ..., \beta_{\sigma_{\infty}(N)} \in \CC$ such that \begin{equation}\label{e:basis}\piholstar(f_E \cdot \widehat{\Zed}_E) = \alpha f_E + \sum_i \beta_i F_{N, 2}^{\rho(i)}, \end{equation} $F_{N, 2}^{\rho(i)}$ is the weight $2$ quasimodular form for $\Gamma_0(N)$ that takes the value $1$ at the cusp $\rho(i)$ and vanishes at all other inequivalent cusps.
\end{cor}

\section{Proofs of theorems}

Throughout this section, let $E$ be a strong Weil elliptic curve with conductor $N_E$ and associated weight $2$ newform $f_E = \sum_{n = 1}^{\infty} a_E(n)q^n$ of level $N = N_E$, where $\text{genus}(X_0(N)) = 1$.

\begin{defn}
Consider $E$ as defined above. For each $q | N_E$, $q \in \ZZ^+$, we define the \emph{Atkin-Lehner involution} $W_q$ to be $$W_q = \begin{bmatrix}
q^{\alpha} a & b \\
N_E c & q^{\alpha} d \\
\end{bmatrix}$$ where $q^{\alpha} || N_E$ and $a, b, c, d \in \ZZ$ such that $W_q$ has determinant $q^{\alpha}$.
\end{defn}

We can use these $W_q$ to understand the Fourier expansion of $\widehat{\Zed}_E$ and ultimately the holomorphic projection $\piholstar(f_E \cdot \widehat{\Zed}_E)$ at the $\Gamma_0(N)$ inequivalent cusps.

\begin{prop}[\S 1 \cite{ALF15}]\label{p:lambda}
Consider $E$, $f_E$ as above. For all $W_q$, there exists $\lambda_q \in \{\pm 1\}$, the \textit{Atkin-Lehner eigenvalue of} $f_E$ such that $f_E |_2 W_q = \lambda_q f_E$. 
\end{prop}

Given this characterization of the Atkin-Lehner involution, we can understand the Fourier expansion of $\widehat{\Zed}_E$ at the cusps:

\begin{thm}[\S 1 \cite{ALF15}]\label{p:computeInvolution}
Consider $E$, $f_E$ as above. If $N$ is squarefree and $q |N_E$, then $$\widehat{\Zed}_E |_0 W_q = \widehat{\Zed}_E^+(\lambda_q (\mcE_E(z) - \Omega_q(f_E))) - \frac{1}{4\pi ||f_E||^2} \cdot \overline{\lambda_q (\mcE_E(z) - \Omega_q(f_E))}$$ $$\Omega_q(f_E) = -2\pi i \int_{W_q^{-1}\infty}^{\infty} f_E(z) dz$$
\end{thm}

Consider some cusp representative $\rho(i)$ of $\Gamma_0(N)$. Then, define $\sigma_i$ (with associated $\lambda_i, \Omega_i$) to be the Atkin-Lehner involution such that $\sigma_i \rho(i) = \infty$. Using this, we now prove the following lemmas that will enable us to show Theorem~\ref{t:thm1}.

\begin{lemma}\label{p:good}
Consider $E$ as defined above. The weak harmonic Maa{\ss} form $\widehat{\Zed}_E(z)$
is good for $f_E(z)$, where $f_E(z)$ is the weight $2$ newform associated to $E$.
\end{lemma}
\begin{proof}
Note that as per Lemma~\ref{p:pole}, $\widehat{\Zed}_E^+$ has no poles on the upper half plane and consequently, $M_E(z) = 0$ in the case we consider (where $\genus(X_0(N)) = 1$). If $E$ has CM, then $\widehat{\Zed}_E(z)$ is good for $f_E$ by Theorem 6 of \cite{CL16}. If not, $N$ is squarefree. Note that by construction of $\widehat{\Zed}_E(z)$, this harmonic Maa\ss\, form satisfies condition (3) to be good. 

For squarefree conductor $N$, the Weierstrass mock module form has a pole arising from the $\frac{1}{z}$ term in the Laurent expansion of the evaluation of the Weierstrass $\zeta$ function at $\mcE_{f_E}(z)$. This has no constant term, and thus $\widehat{\Zed}_E^+(z) = q^{-1} + O(1)$.

Following the proof in \cite{ALF15} of Theorem~\ref{p:computeInvolution}, $$\Omega_q(f_E) = \mcE_E(z) - \lambda_q \mcE_E(W_q z), \quad z \in \mcH$$
Therefore, by choosing $z \in \mcH$ so that $\Im(z), \Im(W_q z)$ are comparable (to give tight approximations), we can evaluate the Eichler integrals to verify $\Omega_q(f_E)$ is nonvanishing for each of the finitely many desired cases ($N_E = 11, 14, 15, 17, 19, 21$).

For example, in the case of conductor $N_E = 14$, $$W_{14} = \begin{bmatrix}
0 & -1 \\
14 & 0
\end{bmatrix}, \quad z = \frac{1}{\sqrt{14}} i$$
gives $\Omega_{14} = 0.3302$, $$W_{7} = \begin{bmatrix}
7 & 1 \\
42 & 7
\end{bmatrix}, \quad z = 0.1032 i$$
gives $\Omega_{7} = 0.6862$, and $$W_{2} = \begin{bmatrix}
4 & 1 \\
14 & 4
\end{bmatrix}, \quad z = 0.1091 i$$
gives $\Omega_{2} = -1.3255$.

Following Theorem~\ref{p:computeInvolution}, $\widehat{\Zed}_E |_0 W_q$ has mock modular form contribution $\widehat{\Zed}_E^+(\lambda_q (\mcE_E(z) - \Omega_q(f_E)))$. Since $\Omega_q(f_E)$ is never zero for $q | N$, $\widehat{\Zed}_E(z) = c + O(q)$, for some constant $c$ as desired, and thus  $\widehat{\Zed}_E(z)$ is good for $f_E$.
\end{proof}

\begin{lemma}\label{l:zvanish}
Consider $E$ as defined above. Then $\widehat{\Zed}_E^+$ vanishes at all cusps not equivalent to the cusp  $\infty$ of $\Gamma_0(N)$.  
\end{lemma}
\begin{proof}
If $E$ satisfies the above conditions, then $\widehat{\Zed}_E$ is good for $f_E$. In particular, $\widehat{\Zed}_E$ has a pole at $\infty$ and constant principal part at the cusps of $\Gamma_0(N)$ by Lemma~\ref{p:good}. 
On the other hand, the index $-1$ Maa{\ss}-Poincar\'e series $Q(-1, 2,N;z)$ has a pole at $\infty$ and zero principal part at the cusps (Proposition~\ref{p:poincusp}). 
The difference $\widehat{\Zed}_E(z) - Q(-1, 2, N;z)$ is a weight $0$ harmonic Maass form with no poles and constant value at each of the cusps. 
The differential operator $\xi_0$ maps harmonic Maa{\ss} forms of weight $2-k$ to cusp forms of weight $k$. Then since the dimension of $S_2(\Gamma_0(N))$ is $1$, we have
\[\xi_0(\widehat{\Zed}_E(z) - Q(-1, 2, N;z)) = c \cdot f_E(z) \]
for some constant $c$. However, by Lemma~\ref{p:nonconstant}, since $\widehat{\Zed}_E(z) - Q(-1, 2, N;z)$ has constant principal part at all cusps, we find that $c = 0$. This implies that the difference is holomorphic, and so is constant. Then  $\widehat{\Zed}_E(z)$ and $Q(-1, 2, N;z)$ are equal up to an additive constant since both have leading term $q^{-1}$, as are their holomorphic parts $\widehat{\Zed}_E^+(z)$ and $Q^+(-1, 2, N;z)$. Thus, as in \cite{RH12}, since $Q^+(-1, 2, N; z)$ vanishes at all cusps not equivalent to infinity by Lemma~\ref{p:poincusp}, $\widehat{\Zed}_E^+(z)$ vanishes at all cusps not equivalent to infinity.
\end{proof}

Using this Lemma, we are able to prove the first theorem.

\begin{proof}[Proof of Theorem~\ref{t:thm1}]
From Lemma~\ref{p:picomputation} and Lemma~\ref{l:holprojRHS}, we obtain that $$\piholstar(\widehat{\Zed}_E \cdot f_E)(z) = f_E(z) \cdot \widehat{\Zed}_E^+(z) - \frac{\pi}{\vol(\Lambda_E)} \sum_{h = 1}^{\infty} D_{f_E}(h; 1) q^h = \alpha f_E + \sum_i \beta_i F_{N, 2}^{\rho(i)}(z).$$

We can compute the holomorphic projection at each cusp $\rho(i)$ using Atkin-Lehner involutions as in Theorem~\ref{p:computeInvolution} (since $N$ here is squarefree), which gives the values $\beta_i$. 
Note that the $L$-series generating function $\mathbb{L}_{f_E}(z)$ vanishes at the cusp $\infty$ and $f_E$ vanishes at all cusps $\rho(i)$.
Applying Lemma~\ref{l:zvanish}, $\widehat{\Zed}_E^+$ vanishes at all cusps inequivalent to the  cusp $\infty$ of $\Gamma_0(N)$. Consequently the holomorphic projection vanishes at all cusps not $\Gamma_0(N)$-equivalent to $\infty$ and is $1$ at the cusp $\infty$. If we let $\rho(1)$ denote the cusp $\infty$, $\beta_i = 0$ for $i \neq 1$ and $\beta_1 = 1$. Using this, we can compute $\alpha$ by equating the first Fourier coefficient of both expressions for $\piholstar$. Rearranging gives the desired expression for $\mathbb{L}_{f_E}(z)$.
\end{proof}

We also consider the case where $E$ as defined in the beginning has complex multiplication, which will  give an analogous result for the shifted convolution $L$-series values.

\begin{lemma}\label{p:lservanish}
Consider $E$ as defined above which also has complex multiplication and thus conductor $N \neq 49$ with associated modular form $f_E(z)$. Then $ D_{f_E}(h; 1) = 0$ if $h \not \equiv 0 \pmod{n_0}$. 
\end{lemma}
\begin{proof}
Suppose that $E$ has complex multiplication. Then if $p$ is a prime inert in the CM field, $a_E(p) = 0$. Since $f_E$ is a weight $2$ newform, its coefficients are multiplicative. As in \S5 of~\cite{BRU08}, $a_E(n) = 0$ for all $n \not \equiv 1 \pmod{n_0}$ where $n_0 | N$ is a curve-dependent value always at least $3$. For example, in the case of the $\Gamma_0(27)$-optimal elliptic curve with complex multiplication, $n_0 = 3$. Now recall that $D_{f_E}(h; 1)$ is defined by
\[ D_{f_E}(h; 1) = \sum_{n = 1}^\infty a_E(n+h)a_E(n) \left(\frac{1}{n+h}-\frac{1}{n} \right). \]
Suppose that $ D_{f_E}(h; 1)$ is nonvanishing. Then both $a_E(n+h)$ and $a_E(n)$ must be nonvanishing and thus $n+h, n \equiv 1 \pmod{n_0}$. This yields $h \equiv 0 \pmod{n_0}$. 
\end{proof}

\begin{rem}
For $N = 49$, the support of the Hecke eigenvalues $a_E(n)$ are at $n \equiv 1, 2, 4 \pmod{7}$. In this case, the proof does not hold, and the $L$-series has support everywhere. Thus, although the theorem appears to hold numerically in the $N = 49$ case, this approach does not yield the desired result.
\end{rem}

\begin{lemma}\label{l:piholvanish}
Consider $E$ with complex multiplication as defined at the beginning of the section with conductor $N \neq 49$ and associated modular form $f_E$. Then there exists some $n_0 \ge 3$ with $n_0 | N$ such that $\piholstar(f_E \cdot \widehat{\Zed}_E)[h] = 0$ if $h \not \equiv 0 \pmod{n_0}$. 
\end{lemma}
\begin{proof}
The derivative of the Weierstrass mock modular form $\widehat{\Zed}_E^+$ for the three strong Weil curves of conductor $N = 27, 32,$ and $36$ are given as eta-quotients in the following table (see \cite{CL16}). 

\begin{center}
\begin{tabular}{ |c||c|c| } 
 \hline
 $N$ &  $\widehat{\Zed}_E^+( \cdot )$ & $q \frac{d}{dq} (\widehat{\Zed}_E^+)$ \\ \hline 
 $27$ &  $q^{-1} + \dfrac12 q^2 + \dfrac15 q^5 + \dfrac34 q^8 + \cdots$ & $- \dfrac{ \eta(3\tau)\eta(9\tau)^6}{\eta(27\tau)^3}$ \\ \hline
 $32$ &  $q^{-1} + \dfrac23 q^3 + \dfrac17 q^7 - \dfrac{2}{11} q^{11} + \cdots$ & $- \dfrac{\eta^2(4\tau) \eta^6(16\tau)}{ \eta^{4}(32\tau)}$  \\ \hline
 $36$ &  $q^{-1} + \dfrac35 q^5 + \dfrac{1}{11} q^{11} + \cdots$ & $-\dfrac{\eta^3(6\tau) \eta(12\tau) \eta(18\tau)}{\eta^3(36\tau)}$ \\ \hline
\end{tabular}
\end{center}

The support of the derivative $q \frac{d}{dq}(\widehat{\Zed}_E^+)$ is the same as the support of the Weierstrass mock modular form. The form of the eta-quotients indicates that the support for $\widehat{\Zed}_E^+$ for $N = 27$ is $-1 \pmod{3}$, for $N = 32$ is $-1 \pmod{4}$, and for $N = 36$ is $-1 \pmod{6}$. Let $n_0 = 3, 4, 6$ for $N = 27, 32, 36$ respectively.

Each of these curves has complex multiplication, and their associated modular forms $f_E$ have support (are nonvanishing) only at $1 \pmod{n_0}$. Thus the product of the Weierstrass mock modular form $\widehat{\Zed}_E^+$ and $f_E$ has support at $0 \pmod{n_0}$. 

By Lemma~\ref{p:lservanish}, the $L$-series is also only supported at $0 \pmod{n_0}$, which yields that the holomorphic projection
\[\pihol(f_E \cdot \widehat{\Zed}_E^+) = f_E \cdot \widehat{\Zed}_E^+ - \frac{\pi}{\vol(\Lambda_E)} \mathbb{L}_{f_E}(z) \]
has support only at $0 \pmod{n_0}$. 
\end{proof}

Now, we can prove Theorem~\ref{t:thm1.2}.
\begin{proof}[Proof of Theorem~\ref{t:thm1.2}]
Following the proof for Theorem~\ref{t:thm1}, we obtain that $\beta_1 = 1$, and $\beta_i = 0$ for $i \neq 1$ in Equation~\eqref{e:basis}. Since $f_E$ is only supported at coefficients $1$ mod $n_0$, where the left hand side is not supported (see Lemma~\ref{l:piholvanish}), $\alpha = 0$. 
Thus, we have that $$\piholstar(\widehat{\Zed}_E \cdot f_E)(z) = F_{N, 2}^{\infty}(z).$$
We can compute the desired closed form expression for the shifted-convolution $L$-series values: \[\mathbb{L}_{f_E}(z) = \frac{\vol(\Lambda_E)}{\pi}\big((f_E(z) \cdot \widehat{\Zed}_E^+(z)) - F^\infty_{N,2}(z)  \big). \]
\end{proof}

\begin{rem}
Note that results such as the above theorems may be able to be transformed to give recurrence relations on the modular form coefficients, giving expressions for the Fourier coefficients analogous to those obtained in obtained in $\S 12$ of \cite{Ono09} and \cite{CHO10}. 
\end{rem}

\section{Examples}\label{examples}

\subsection{Conductor $N = 11$}
Consider the modular curve $X_0(11)$ of dimension $1$. There is a single isogeny class of elliptic curves, and the strong Weil curve is given by the Weierstrass equation
\[E: y^2 + y = x^{3} -  x^{2} - 10 x - 20. \]
Numerically, we find that $S(\Lambda_E) = 0.38124\dots $. Using this and the Fourier expansion of the Weierstrass $\zeta$-function, the corresponding weight $0$ mock modular form $\widehat{\Zed}_E^+(z)$ is given by
\[ q^{-1} + 1 + 0.9520\dots q + 1.547\dots q^2 + 0.3493\dots q^3 + 1.976\dots q^4 -2.609\dots q^5 + O(q^6).\]
Using the formula given in Proposition~\ref{p:picomputation}, one can compute the $L$-series numerically. Using $100000$ coefficients of $f_E$, we have the Fourier expansion $\mathbb{L}_{f_E}(z) = \sum_{n = 1}^\infty D_{f_E}( h; 1)q^h$:
\begin{equation}\label{e:11L} \mathbb{L}_{f_E}(z) = 0.7063\dots q + 1.562\dots q^2 + 0.0944\dots  q^3 + 1.237\dots q^4 - 2.026\dots q^5 + O(q^6)  \end{equation}
The space of weight $2$ quasimodular Eisenstein series for $\Gamma_0(11)$ is $2$ dimensional, and a basis is given by $F_{11,2}^0$ and $F_{11,2}^\infty$, where $F_{11,2}^\infty$ is
\begin{align*}
F_{11,2}^\infty &= 1 + \frac15q + \frac35 q^2 + \frac45 q^3 + \frac75 q^4 + \frac65q^5 + \frac{12}{5}q^6 + O(q^7).
\end{align*}
Then, if we take $\alpha = .0016, \beta_1 = 1, \beta_2 = 0$, we obtain another way to retreive the $L$-series values in Equation~\eqref{e:11L}:
\begin{equation*} \begin{split}
\frac{\vol(\Lambda_E)}{\pi}\left((f_E \cdot \widehat{\Zed}_E^+) - \alpha f_E - F^{\infty}_{11,2} \right)  &= 
\\ 0.706\dots q + 1.562\dots q^2 + 0.0930\dots q^3 &+ 1.234\dots q^4 - 2.024\dots q^5 + O(q^6).
\end{split}\end{equation*}
Note that to the accuracy of the computations, the value $\alpha \approx 0$ in this case.
 
\subsection{Conductor $N = 27$}
 Recall the strong Weil curve of conductor $27$ given by the Weierstrass equation $E_{27}: y^2 + y = x^3 - 7$ (Cremona label 27a1). 
The weight $2$ modular form associated with $E_{27}$ is given by
\[f_{E_{27}} =  q - 2q^{4} - q^{7} + 5q^{13} + 4q^{16} - 7q^{19} + O(q^{20})  .\]
Using the Fourier expansion of the Weierstrass $\zeta$-function, the weight $0$ mock modular form associated to $\widehat{\Zed}_{E_{27}}^+(z)$ is given by 
\begin{equation} \label{e:zed27} \widehat{\Zed}_{E_{27}}^+(z) = q^{-1} + \frac{1}{2}q^2 + \frac{1}{5}q^5 + \frac{3}{4}q^8 - \frac{6}{11}q^{11} - \frac{1}{2}q^{14} + O(q^{17}) .\end{equation}  
 The holomorphic projection is given by
 \[\piholstar(f_E \cdot \widehat{\Zed}_{E_{27}}^+)(z) = 1 + 3q^9+  9q^{18} - 12q^{27} \dots. \]
 On the other hand, we find that the normalized element of the weight $2$ quasimodular forms for $\Gamma_0(27)$ that vanishes at all cusps but $\infty$ is given by
 \[F_{27,2}^\infty = 1 + 3q^9+  9q^{18} - 12q^{27} \dots \]
 This agrees with the computation for the holomorphic projection.
 Then the generating function $\mathbb{L}_f(z)$ can also be computed by Theorem~\ref{t:thm1.2}, giving arbitrary precision computations for these slowly convergent shifted convolution $L$-series values that could only previously be computed term-by-term.

\section*{Acknowledgements}

The authors would like to thank the NSF (grant DMS-1250467) and the Emory REU (especially Dr. Mertens and Professor Ono) for their support. We would also like to thank the anonymous referee for their careful review and helpful suggestions.

\end{document}